%% file: A_dual_algorithm.tex
\documentclass[10pt]{article}
\usepackage[left=1in,top=1in,right=1in,bottom=1in,letterpaper]{geometry}

\input{macros}
\begin{document}

\title{A dual algorithm for a class of augmented convex models}

\author{
Hui Zhang\thanks{Department of Mathematics and Systems Science,
College of Science, National University of Defense Technology,
Changsha, Hunan, China, 410073. Email: \texttt{hhuuii.zhang@gmail.com}}
\and Lizhi Cheng$^*$
\and Wotao Yin\thanks{Department of Mathematics, University of California, Los Angeles, CA. Email: \texttt{wotaoyin@math.ucla.edu} }}
\date{\today}

\maketitle

\begin{abstract}
Convex optimization models find interesting applications, especially in signal/image processing and compressive sensing. We study some  augmented convex models, which are perturbed by strongly convex functions, and propose a dual gradient algorithm. The proposed algorithm includes the linearized Bregman algorithm and the singular value thresholding algorithm as special cases. Based on fundamental properties of proximal operators, we present a concise approach to establish the convergence of both primal and dual sequences, improving the results in the existing literature. 
\end{abstract}


\section{Introduction}
The past two decades have witnessed several successful convex models for signal processing. They include, but are not limited to, the total variation model \cite{rof} and the basis pursuit model \cite{cds}, both of which have been widely applied in signal/image processing and compressed sensing. Recently, augmented convex models \cite{z1,z2,ly}, obtained by adding strongly convex perturbations to the original objective functions, are introduced for fast computation \cite{z3} and for incorporating certain prior information regarding the underlying signal \cite{zh}. Well known in convex analysis \cite{r}, if the original problem is strongly convex, then the dual problem is differentiable and can thus take advantages of  a rich set of gradient-based optimization techniques. In addition, the augmented term  can reflect certain structures of  the target signal, for example, the group structure through the term $\|\cdot\|_2^2$ in the elastic net model \cite{zh}.

This paper is devoted to analyzing a dual gradient algorithm for a class of augmented convex models. 
The proposed algorithm is motivated by and includes two well-known algorithms as its special cases: the linearized Bregman algorithm (LBreg) \cite{y1,y2} and the singular value thresholding algorithm (SVT) \cite{c3}. Based on some fundamental properties of proximal operators, we prove the convergence of both the primal and dual point sequences.  {This result is stronger than the previously shown vanishing distance between the dual sequence and the dual solution set (the dual sequence itself is not shown to converge).}

\section{Augmented convex models for signal recovery}

{Let $x\in\RR^n$.} Consider the convex problem
\begin{equation}\label{Nrm}
\min  \mu\|x\|, \quad\st~ \mathcal{A} x=b,
\end{equation}
and its augmented model
\begin{equation}\label{Aug}
\min  P(x)\triangleq\mu\|x\| +\frac{\mu}{2\tau} \|x\|_2^2, \quad\st~ \mathcal{A} x=b,
\end{equation}
where $\|\cdot\|$ is a norm whose dual norm is denoted by $\|\cdot\|_\diamondsuit$, $\frac{\mu}{2\tau} \|x\|_2^2$ is the \emph{augmented} term, linear operator $\mathcal{A}:\RR^n\to\RR^m$ and observed data $b\in\RR^m$ are given, and $\tau, \mu$ are positive parameters. Throughout the paper, we assume that  $\mathcal{A} x=b$ is consistent. Parameter $\mu$ is redundant to both objectives and does not affect the solutions, but it is kept in order to unify the models and algorithms that appear in the previous literature. Parameter $\tau$ weights  the term $\|\cdot\|_2^2$ and  affects the solution to \eqref{Aug} when it falls in a certain range. In what follows, we give a few examples of \eqref{Aug} in signal recovery. 

\textit{Example 1} (Augmented $\ell_1$ norm and nuclear-norm models)  Paper \cite{ly} proposes the augmented $\ell_1$ norm model for sparse signal recovery
\begin{equation}\label{AugL1}
\min    \|x\|_1 +\frac{1}{2\tau} \|x\|_2^2, \quad\st~ A x=b
\end{equation}
and the augmented nuclear-norm model for low-rank matrix recovery
\begin{equation}\label{AugNu}
\min   \|X\|_* +\frac{1}{2\tau} \|X\|_F^2, \quad\st~ \mathcal{A}(X)=b.
\end{equation}
Suppose $b=Ax^0$ where  $x^0$ is a sparse vector. Model \eqref{AugL1} will recover $x^0$ provided that  $\tau\geq 10\|x^0\|_\infty$  and the sensing matrix $A$ satisfies  certain conditions such as the null-space property and restricted isometry property. Similarly,  $\tau\geq 10\|X^0\|$ is used for recovering a low-rank matrix $X^0$, where $\|X^0\|$ is its spectral norm.

\textit{Example 2} (Strongly convex matrix completion model) Papers \cite{z1,z2} study the following strongly convex model for matrix completion
\begin{equation}\label{Str1}
\min    \|X\|_* +\frac{1}{2\tau} \|X\|_F^2, \quad\st~ \mathcal{P}_\Omega(X) =\mathcal{P}_\Omega(M),
\end{equation}
where $M$ is a low-rank matrix, $\Omega$ is the sample index set, and  $\mathcal{P}_{\Omega}$ is the corresponding element-selection operator. To recover the low-rank matrix $M$, the best known  bound is $\tau\ge\frac{4}{p}\|\mathcal{P}_{\Omega}(M)\|_F$ given in \cite{z2}, where $p$ is the sample ratio.

\textit{Example 3} (Strongly convex RPCA model) Paper \cite{z2} studies the following strongly convex model for robust principle component analysis (RPCA):
\begin{equation}\label{Str2}
\min_{L,S}    \|L\|_*+\frac{1}{2\tau}\|L\|^2_F+\lambda\|S\|_1+\frac{1}{2\tau}\|S\|^2_F, \quad\st~ D=L+S,
\end{equation}
where $D$ is the observed data matrix and $\lambda$ is a given parameter. {Bound $\tau\ge \frac{8\sqrt{15}\|D\|_F}{3\lambda}$ guarantees the  decomposition of the observed matrix $D$ into its low-rank component and sparse component \cite{ywl}.}

Below section \ref{sc:dualalg} presents a dual  gradient algorithm for problem \eqref{Aug} and section \ref{sc:convg} studies its convergence. Section \ref{sc:gauge} extends these results to problems with guage objective functions.

\section{A dual gradient algorithm}\label{sc:dualalg}
In this section, we first introduce  properties of proximal operators required for convergence analysis. Then, we derive a Lagrange dual problem and  an iterative gradient algorithm for solving it. 
\subsection{Proximal operators}
Let $f: \RR^n\rightarrow R\cup\{+\infty\}$ be a closed proper convex function. The proximal operator \cite{m}   $\prox_f :\RR^n\rightarrow \RR^n$ is defined by
\begin{equation}
\prox_f(v)=\argmin_x\left(f(x)+\frac{1}{2}\|x-v\|_2^2\right).
\end{equation}
Since the objective function is strongly convex and proper, $\prox_f(v)$ is properly defined for every $v\in\RR^n$. The following properties \cite{m,pb} will be used in our analysis.
 \begin{lemma}\label{lem1}
 Let $f: \RR^n\rightarrow R\cup\{+\infty\}$ be a closed proper convex function. Then, for all $x, y\in \RR^n$ the proximal operator $\prox_{f(\cdot)}$ satisfies the followings:
 \begin{enumerate}
  \item Firmly nonexpansive: $ \|\prox_{f(\cdot)}(x)-\prox_{f(\cdot)}(y)\|_2^2\leq \langle x-y, \prox_{f(\cdot)}(x)-\prox_{f(\cdot)}(y)\rangle$

  \item Lipschitz continuous: $\|\prox_{f(\cdot)}(x)-\prox_{f(\cdot)}(y)\|_2\leq \|x-y\|_2$
  \end{enumerate}
 \end{lemma}

 \begin{lemma}\label{lem2}
 For any $\tau>0$ and norm $\|\cdot\|$, it holds that $\tau\cdot\prox_{\|\cdot\|}(\frac{1}{\tau}v)=\prox_{\tau\|\cdot\|}(v)$
 \end{lemma}
 \begin{proof}
 Let $u=\prox_{\tau\|\cdot\|}(v)=\argmin_z \tau\cdot \|z\|+\frac{1}{2}\|z-v\|_2^2$.
 Let $w=\prox_{\|\cdot\|}(\frac{1}{\tau}v)$ and $\bar{u}=\tau w$. Since
 \begin{subequations}
\begin{align}
w & =\argmin \|x\|+\frac{1}{2}\|x-\frac{1}{\tau}v\|_2^2 \nonumber \\
&=\argmin \tau\cdot \|\tau x\|+\frac{1}{2}\|\tau x-v\|_2^2 \nonumber
\end{align}
\end{subequations}
after the change of variable $\tau x\to z$, we have $u=\tau w=\bar{u}$. This completes the proof.
\end{proof}

\begin{remark}
 Lemma \ref{lem2} remains valid if $\|\cdot\|$ is replaced by a closed proper convex function $f$ that is one-homogeneous.  We, however, restrict our attention to  $f(\cdot)=\|\cdot\|$ for brevity. 
\end{remark}

 \begin{lemma}\label{lem3}
 Moreau decomposition:  any $v\in \RR^n$ can be decomposed as $v=\prox_f(v)+\prox_{f^*}(v)$, where $f^*=\sup_x(\langle y, x\rangle -f(x))$ is the convex conjugate of $f$.
 \end{lemma}
There is a close relationship between  proximal and projection operators.
Let $\mathcal{B}=\{z: \|z\|_\diamondsuit\leq 1\}$ and consider the projection onto $\mathcal{B}$: $\Pi_\mathcal{B}(v)=\argmin_{x\in \mathcal{B}}\|x-v\|_2$.  Applying the Moreau decomposition to $\|\cdot\|$, we have
\begin{equation}\label{Moreau}
v=\prox_{\|\cdot\|}(v)+\Pi_\mathcal{B}(v).
\end{equation}
In order to derive the gradient of  $D(y)$,  we define the point-to-set function
\begin{equation}\label{hfun}
h_\mathcal{Z}(x)=\min_{z\in \mathcal{Z}}\|x-z\|_2,
\end{equation}
where $\mathcal{Z}$ is a closed convex set.
Following Example 2.79 in \cite{ru}, it holds that
 \begin{eqnarray}
h_\mathcal{Z}(x)=\left\{\begin{array}{ll}
\|x-\Pi_\mathcal{Z}(x)\|_2~~&x\notin \mathcal{Z} \\
0~~&x\in\mathcal{Z},
\end{array} \right.
\end{eqnarray} and
 \begin{eqnarray}
 \label{gradh}
\nabla h_\mathcal{Z}(x)=\left\{\begin{array}{ll}
\frac{x-\Pi_\mathcal{Z}(x)}{\|x-\Pi_\mathcal{Z}(x)\|_2}~~&x\notin \mathcal{Z} \\
0~~&x\in\mathcal{Z}.
\end{array} \right.
\end{eqnarray}

\subsection{Lagrange dual analysis}\label{sec:3}
The Lagrangian of the augmented convex model (\ref{Aug}) is
\begin{equation}
L(x, y)=\mu\|x\| +\frac{\mu}{2\tau} \|x\|_2^2 +\langle y, b-\mathcal{A}x\rangle.
\end{equation}
Following $\|x\|=\max_{\|z\|_\diamondsuit\leq 1}\langle x, z\rangle$, we derive the dual function as
\begin{subequations}
\begin{align}
L_D(y)=\min_x L(x, y)& =\langle y, b\rangle +\min_x\max_{\|z\|_\diamondsuit\leq 1}\mu\langle x, z\rangle+\frac{\mu}{2\tau}\|x\|_2^2-\langle \mathcal{A}^*y, x\rangle \nonumber \\
& =\langle y, b\rangle +\max_{\|z\|_\diamondsuit\leq 1}\min_x\mu\langle x, z\rangle+\frac{\mu}{2\tau}\|x\|_2^2-\langle \mathcal{A}^*y, x\rangle \nonumber \\
&=\langle y, b\rangle -\frac{\tau}{2\mu} \min_{\|z\|_\diamondsuit\leq 1} \|\mathcal{A}^*y-\mu z\|_2^2  \nonumber\\
&=\langle y, b\rangle -\frac{\tau\mu}{2} \min_{\|z\|_\diamondsuit\leq 1} \|\frac{1}{\mu}\mathcal{A}^*y-  z\|_2^2  \nonumber
\end{align}
\end{subequations}
where the $x$-minimization problem has solution $x=\frac{\tau}{\mu}(\mathcal{A}^*y-\mu z)$. Hence, the dual problem is
\begin{equation}\label{Dual}
\max_y L_D(y)=-\min D(y),\quad \text{where}~D(y)\triangleq -\langle y, b\rangle +\frac{\tau\mu}{2} \min_{\|z\|_\diamondsuit\leq 1} \|\frac{1}{\mu}\mathcal{A}^*y -  z\|_2^2.
 \end{equation}
 The minimum over  $z$ is obtained at $z=\Pi_\mathcal{B}(\frac{1}{\mu}\mathcal{A}^*y)$.
To distinguish $L_D(y)$ and $D(y)$, we call the later dual objective. Following the definition of $h_\mathcal{Z}(x)$ in (\ref{hfun}), $D(y)$ can be written equivalently as
$$D(y)= -\langle y, b\rangle +\frac{\tau\mu}{2} h_\mathcal{B}(\frac{1}{\mu}\mathcal{A}^*y )^2.$$
Following from \eqref{gradh}, the gradient of $D(y)$ is
\begin{equation}
\nabla D(y)= -b +\tau \mathcal{A} \left(\frac{1}{\mu}\mathcal{A}^*y-\Pi_\mathcal{B}(\frac{1}{\mu}\mathcal{A}^*y)\right),
\end{equation}
and, due to \eqref{Moreau},
\begin{equation}\label{gradD}
\nabla D(y)= -b +\tau\mathcal{A} \cdot\prox_{\|\cdot\|}\left(\frac{1}{\mu}\mathcal{A}^*y\right).
\end{equation}
We highlight the primal-dual relationship:
$x=\frac{\tau}{\mu}(\mathcal{A}^*y-\mu z)$ and
 $z=\Pi_\mathcal{B}(\frac{1}{\mu}\mathcal{A}^*y)$. If $y$ is dual optimal, by standard convex analysis,  $x$ in the relationship is primal optimal.
\begin{lemma}\label{doptset}
 Let $\hat{x}$ be the unique solution to problem \eqref{Aug}. Then the dual solution set to problem \eqref{Dual} is
\begin{equation}\label{dualY}
\mathcal{Y}=\left\{y: \tau\cdot \prox_{\|\cdot\|}(\frac{1}{\mu}\mathcal{A}^*y)=\hat{x}\right\},
\end{equation}
  which is  nonempty and convex.
  \end{lemma}
   \begin{proof}
     From convex analysis and \eqref{gradD} it follows that the dual solution set is $\mathcal{Y}^{'}=\{y: \nabla D(y)=0\}=\{y: \tau\mathcal{A}\cdot \prox_{\|\cdot\|}(\frac{1}{\mu}\mathcal{A}^*y)=b\}$. Comparing this to \eqref{dualY} and since $\cA\cdot\hat{x}=b$, we have $\mathcal{Y}\subset \mathcal{Y}^{'}$. Therefore, it suffices to show $\mathcal{Y}^{'}\subset \mathcal{Y}$. Indeed, let $\hat{y}\in \mathcal{Y}^{'}$; following the primal-dual relationship, $\hat{y}$ shall give optimal $\hat{x}$, i.e.,
  $$\frac{\tau}{\mu}\left(\mathcal{A}^*\hat{y}-\mu \Pi_\mathcal{B}(\frac{1}{\mu}\mathcal{A}^*\hat{y})\right)=\frac{\tau}{\mu}(\mathcal{A}^*\hat{y}-\mu \hat{z})=\hat{x}.$$
Since the left-hand side equals, $\tau\cdot \prox_{\|\cdot\|}(\frac{1}{\mu}\mathcal{A}^*\hat{y})$,
by the definition $\cY$, we have $\hat{y}\in \mathcal{Y}$.
  The convexity of $\mathcal{Y}$ follows from the convexity of primal problem; it can also be seen through:
  \begin{subequations}
\begin{align}
\langle \nabla D(y_1)-\nabla D(y_2), y_1-y_2\rangle &=\tau \mu \langle \prox_{\|\cdot\|}(\frac{1}{\mu}\mathcal{A}^*y_1)-\prox_{\|\cdot\|}(\frac{1}{\mu}\mathcal{A}^*y_2),\frac{1}{\mu}\mathcal{A}^*y_1-\frac{1}{\mu}\mathcal{A}^*y_2\rangle \\
\geq&\tau \mu \left\|\prox_{\|\cdot\|}(\frac{1}{\mu}\mathcal{A}^*y_1)-\prox_{\|\cdot\|}(\frac{1}{\mu}\mathcal{A}^*y_2) \right\|_2^2\geq 0,
\end{align}
\end{subequations}
where the inequality follows from Lemma \ref{lem1}.
The consistency of $\mathcal{A}x=b$ guarantees $\mathcal{Y}$ to be nonempty.
\end{proof}

\subsection{Algorithm and examples}
Applying the gradient iteration to the dual objective $D(y)$ gives:
\begin{equation}\label{form1}
y^{k+1}=y^k+h\left(b-\tau\mathcal{A} \cdot\prox_{\|\cdot\|}(\frac{1}{\mu}\mathcal{A}^*y^k)\right),
\end{equation}
where $h>0$ is the step size whose range shall be studied later for convergence. By setting $x^{k+1}= \tau\cdot\prox_{\|\cdot\|}(\frac{1}{\mu}\mathcal{A}^*y^k)$, we obtain the
equivalent iteration in the primal-dual form:
\begin{eqnarray}\label{form2}
\left\{\begin{array}{ll}
x^{k+1}= \tau\cdot\prox_{\|\cdot\|}(\frac{1}{\mu}\mathcal{A}^*y^k) \\
y^{k+1}=y^{k}+h(b-\mathcal{A}x^{k+1}).
\end{array} \right.
\end{eqnarray}
Recalling $\tau\cdot\prox_{\|\cdot\|}(\frac{1}{\tau}v)=\prox_{\tau\|\cdot\|}(v)$ from Lemma \ref{lem2} and setting $\mu=\tau$, we simply it to
\begin{eqnarray}\label{form3}
\left\{\begin{array}{ll}
x^{k+1}= \prox_{\tau\|\cdot\|}(\mathcal{A}^*y^k) \\
y^{k+1}=y^{k}+h(b-\mathcal{A}x^{k+1}).
\end{array} \right.
\end{eqnarray}

\textit{Example 1} (The LBreg algorithm) It is a well studied algorithm for solving the augmented $\ell_1$-norm model and has the following primal-dual form
  \begin{eqnarray}\nonumber
\left\{\begin{array}{ll}
x^{k+1}= \tau \cdot \textrm{shrink}(A^Ty^{k}) \\
y^{k+1}=y^{(k)}+h (b-Ax^{k+1})
\end{array} \right.
\end{eqnarray}
where $\shrink(\cdot)$ equals $\prox_{\|\cdot\|_1}(\cdot)$. It is a special case of (\ref{form2}).
The iteration is proposed in \cite{y1} and its convergence analyzed in \cite{c1,c2,y2}. Then paper \cite{ly} establishes its global geometric convergence, whose rate is further improved in \cite{z3}.

\textit{Example 2} (The SVT algorithm)  It is a well-known algorithm for matrix complete and has the following primal-dual form
  \begin{eqnarray}\nonumber
\left\{\begin{array}{ll}
X^{k+1}=  \mathcal{D}_\tau(Y^{k}) \\
Y^{k+1}=Y^{k}+h\cdot\mathcal{P}_\Omega(M-X^{k+1})
\end{array} \right.
\end{eqnarray}
where $\mathcal{D}_\tau(\cdot)$ equals $\prox_{\tau\|\cdot\|_*}(\cdot)$. Generally, we take $Y^0=0$ so that $\mathcal{P}_\Omega(Y^{k})=Y^{k}$ from $Y^{k+1}=Y^{k}+h\cdot\mathcal{P}_\Omega(M-X^{k+1})$. Moreover, $\mathcal{P}_\Omega=\mathcal{P}_\Omega^*$. Hence, it is a special case of (\ref{form3}) with $\mathcal{A}=\mathcal{P}_\Omega$. It is proposed in \cite{c3}.

\begin{remark}
In paper \cite{hmg}, Nesterov's first-order methods \cite{n1} are applied to accelerate the LBreg and  SVT algorithms. Further speedup is introduced in \cite{z3} by combining Nesterov's methods \cite{n1,n2} with an adaptive restart technique \cite{oc}. With little effort, they can be applied to the primal-dual algorithm (\ref{form2}). 
\end{remark}

\section{Convergence analysis}\label{sc:convg}
In this part, we prove the convergence of primal sequence $\{x^k\}$ and dual sequence $\{y^k\}$ in iteration (\ref{form2}).
\begin{theorem}\label{thm:cvg}
Set step size $h\in (0, \frac{2\mu}{\tau\|\mathcal{A}\|^2})$ and $y^0=0$ in iteration (\ref{form2}). Let $\hat{x}$ be the unique minimizer to problem (\ref{Aug}) and $\mathcal{Y}$ be the solution set to problem \eqref{Dual}. Then, $\lim_{k\rightarrow +\infty}x^k=\hat{x}$, and there exists a point $\bar{y}\in\mathcal{Y}$ such that $\lim_{k\rightarrow +\infty} y^k=\bar{y}$.
\end{theorem}
\begin{proof}
Let $\hat{y}\in \mathcal{Y}$. By Lemma \ref{doptset}, we have $\hat{x}=\tau\cdot \prox_{\|\cdot\|}(\frac{1}{\mu}\mathcal{A}^*\hat{y})$. Together with $x^{k+1}= \tau\cdot\prox_{\|\cdot\|}(\frac{1}{\mu}\mathcal{A}^*y^k)$ and Lemma \ref{lem1}, we derive that
 \begin{subequations}
\begin{align}
& \langle \frac{1}{\mu}\mathcal{A}^*y^k-\frac{1}{\mu}\mathcal{A}^*\hat{y}, x^{k+1}- \hat{x}\rangle \\
=&\tau\cdot \langle \frac{1}{\mu}\mathcal{A}^*y^k-\frac{1}{\mu}\mathcal{A}^*\hat{y}, \prox_{\|\cdot\|}(\frac{1}{\mu}\mathcal{A}^*y^k)- \prox_{\|\cdot\|}(\frac{1}{\mu}\mathcal{A}^*\hat{y})\rangle \\
\geq & \tau\cdot \left\|\prox_{\|\cdot\|}(\frac{1}{\mu}\mathcal{A}^*y^k)-\prox_{\|\cdot\|}(\frac{1}{\mu}\mathcal{A}^*\hat{y})\right\|_2^2 \\
=&\tau^{-1}\cdot \| x^{k+1}- \hat{x}\|_2^2
\end{align}
\end{subequations}
Using this inequality, we have
 \begin{subequations}
\begin{align}
\|y^{k+1}-\hat{y}\|_2^2=&\|y^k-\hat{y} +h (b-\mathcal{A}x^{k+1})\|_2^2 \\
=&\|y^k-\hat{y} +h (\mathcal{A}\hat{x}-\mathcal{A}x^{k+1})\|_2^2 \\
=& \|y^k-\hat{y}\|_2^2 -2h\mu \langle \frac{1}{\mu}\mathcal{A}^*y^k-\frac{1}{\mu}\mathcal{A}^*\hat{y}, x^{k+1}- \hat{x}\rangle + h^2\|\mathcal{A}\hat{x}-\mathcal{A}x^{k+1}\|_2^2 \\
\leq & \|y^k-\hat{y}\|_2^2-2h\frac{\mu}{\tau}\| x^{k+1}- \hat{x}\|_2^2+h^2\|\mathcal{A}\|^2\| x^{k+1}- \hat{x}\|_2^2\\
=& \|y^k-\hat{y}\|_2^2-h(\frac{2\mu}{\tau}-h\|\mathcal{A}\|^2)\| x^{k+1}- \hat{x}\|_2^2.
\end{align}
\end{subequations}
Therefore, under the assumption $0<h<\frac{2\mu}{\tau\|\mathcal{A}\|^2}$ we can make the following claims:

\textbf{claim 1:} $\|y^{k+1}-\hat{y}\|_2$ is monotonically nonincreasing in $k$ and thus converges to a limit;

\textbf{claim 2:} $\|x^{k+1}-\hat{x}\|_2 $ converges to 0 as $k$ tends to $+\infty$, i.e., $\lim_{k\rightarrow +\infty} x^{k+1}=\hat{x}$.

From claim 1, it follows that $\{y^k\}$ is bounded and thus has a converging subsequence $y^{k_i}$. Let $\bar{y}=\lim_{i\to\infty}y^{k_i}$. By the Lipschitz continuity of the proximal operator, proved in Lemma \ref{lem1}, we have
$$
\hat{x}=\lim_{i\rightarrow \infty} x^{k_i+1}=\lim_{i\rightarrow \infty} \tau\cdot \prox_{\|\cdot\|}(\frac{1}{\mu}\mathcal{A}^*y^{k_i})=\tau\cdot \prox_{\|\cdot\|}(\frac{1}{\mu}\mathcal{A}^*\bar{y}),
$$
so $\bar{y}\in \mathcal{Y}$ by \eqref{dualY}. Recall $\hat{y}\in\mathcal{Y}$ is arbitrary. Hence, claim 1 holds for $\hat{y}=\bar{y}$. If $\{y^k\}$ had another limit point, then $\|y^{k+1}-\bar{y}\|_2$ would fail to be monotonic. So, $y^k$ converges to $\bar{y}\in\mathcal{Y}$ (in norm).

\end{proof}
\begin{remark} Being a dual gradient algorithm, it is well known that the \emph{dual objective sequence} converges at a rate of $O(1/k)$. With Nesterov's acceleration \cite{hmg}, the rate improves to $O(1/k^2)$. For piece-wise linear norm $\|\cdot\|$, such as the 1-norm, the rate improves to $O(e^{-k})$ and applies to both the  sequence and primal/dual point sequences following the arguments in \cite{ly,z3}.
\end{remark}

\section{Extension to gauge}\label{sc:gauge}
Some  interesting models such as those based on total variation, analysis $\ell_1$, and fused Lasso use objective functions that are related to but more general than norms. To extend our results to these models, we study the gauge objective.
\begin{definition}[Gauge \cite{r}]
Let $C\subset \RR^n$ be a closed convex set containing the origin. The gauge of $C$ is the function $\gamma_C:\RR^n\to \RR$ defined by
$$\gamma_C(x)=\inf \{\lambda>0: x\in \lambda C\}.$$
\end{definition}
If $C$ is bounded and symmetric and has a nonempty interior, then $\gamma_C$ recovers a norm, whose unit ball is $C$. If such $C$ is unbounded, then $\gamma_C$ generalizes to a semi-norm. Recent paper \cite{v} studies (strongly) piecewise regular gauges for signal recovery, which include analysis-type $\ell_1$ semi-norms such as total variation and fused Lasso. We consider a general gauge function $J$ in the following model
\begin{equation}\label{Gau}
\min  J(x), \quad\st~ \mathcal{A} x=b,
\end{equation}
and its augmented model
\begin{equation}\label{Aug1}
\min  P(x)\triangleq  J(x) +\frac{1}{2\tau}\|x\|_2^2, \quad\st~ \mathcal{A} x=b.
\end{equation}



\subsection{Gauge and its polar}
We collect the definitions of the polar set and polar gauge, as well as some useful properties from \cite{r}.
\begin{definition}[Polar set]
 Let $C\subset \RR^n$ be a non-empty closed convex set. The polar of $C$ is
 $$C^o= \{v:\langle v, x\rangle \leq 1,~~\forall x\in C\}.$$
\end{definition}
\begin{definition}[Polar Gauge]
 The polar of a gauge $\gamma_C$ is the function $\gamma^\circ_C:\RR^n\to \RR$ defined by
 $$\gamma^\circ_C(u)=\inf \{\mu\geq0: \langle x, u\rangle \leq \mu\gamma_C(x), \forall x\}.$$
\end{definition}

 \begin{lemma}\label{lem5}
  Let $C\subset \RR^n$ be a closed convex set containing the origin. Then,
   \begin{enumerate}
  \item   $\gamma_C^o=\gamma_{C^o}$, or equivalently $C^o=\{x: \gamma_C^o(x)\leq 1\}=\{x:  \gamma_{C^o}(x)\leq 1\}$, which is a closed convex set.

  \item   $\gamma_C=\sigma_{C^o}$ and $\gamma_{C^o}=\sigma_C$, where $\sigma_C(x)$ is the support function of $C$.

  \item   Moreau decomposition: $v=\prox_{r_C(\cdot)}(v)+\Pi_{C^o}(v)$ for any $v\in \RR^n$.

  \end{enumerate}
\end{lemma}
\begin{proof}
Parts (i) and (ii) are given in corollaries 15.1.1 and 15.1.2 in \cite{r}, respectively. To show part (iii), let $f(x)=\delta_{C^o}(x)$, where $\delta_\mathcal{C}(x)$ is the indicator function
 \begin{eqnarray*}
\delta_\mathcal{C}(x)=\left\{\begin{array}{ll}
+\infty~~&x\notin \mathcal{C} \\
0~~&x\in\mathcal{C}.
\end{array} \right.
\end{eqnarray*}
Then, the convex conjugate of $f(x)$ is
$$f^*(x)=\sup_y(\langle x, y\rangle- \delta_{C^o}(y))=\max_{y\in C^o}\langle x, y\rangle\stackrel{\theta_1}{=}\sigma_{C^o}(x)\stackrel{\theta_2}{=}\gamma_C(x),$$
where $\theta_1$ follows from the definition of support function and $\theta_2$ follows from part (ii). From $\prox_{f(\cdot)}(v)=\Pi_{C^o}(v)$ and $v=\prox_{f(\cdot)}(v)+\prox_{f^*(\cdot)}(v)$, the result follows.
\end{proof}

\subsection{Dual analysis and algorithm}
Let $J(x)=\gamma_C(x)$ be a gauge function. 
Based on part (ii) of Lemma \ref{lem5}, we have $\gamma_C(x)=\sigma_{C^o}(x)=\max_{z\in C^o}\langle x, z\rangle$, from which we  can follow subsection \ref{sec:3} and derive the dual problem of \eqref{Aug1}:
\begin{equation}\label{Dual1}
 \min_y D_J(y):= -\langle y, b\rangle +\frac{\tau}{2} \min_{z\in C^o} \|\mathcal{A}^*y -  z\|_2^2,
\end{equation}
where the optimal $z=\Pi_{C^o}(\mathcal{A}^*y)$ is a function of $y$.
Following the definition of $h_\mathcal{Z}(x)$ in (\ref{hfun}), we have 
$$D_J(y)= -\langle y, b\rangle +\frac{\tau}{2} h_{C^o}(\mathcal{A}^*y )^2.$$
By \eqref{gradh} and part (iii) of Lemma \ref{lem5}, we obtain
\begin{equation}\label{gradD1}
\nabla D_J(y)= -b +\tau\mathcal{A} \cdot\prox_{\gamma_C(\cdot)}(\mathcal{A}^*y)=-b +\tau\mathcal{A} \cdot\prox_{J(\cdot)} (\mathcal{A}^*y).
\end{equation}
We have the primal-dual relationship:
$x=\tau (\mathcal{A}^*y-z)$ and
 $z=\Pi_{C^o}(\mathcal{A}^*y)$. If $y$ is dual optimal, the equations give $x$ that is primal optimal. Similar to Lemma \ref{doptset}, we have the following result:
\begin{lemma}\label{doptset1}
 Let $\hat{x}$ be the unique solution to problem \eqref{Aug1}. Then the dual solution set to problem \eqref{Dual1} is
\begin{equation}\label{dualY1}
\mathcal{W}=\left\{y: \tau\cdot \prox_{J(\cdot)}(\mathcal{A}^*y)=\hat{x}\right\},
\end{equation}
  which is  nonempty and convex.
  \end{lemma}
%
Based on \eqref{gradD1}, one can obtain the dual gradient ascent iteration for problem \eqref{Aug1}. From the above primal-dual relationship, we give the primal-dual form of this algorithm as follows:
  \begin{eqnarray}\label{form5}
\left\{\begin{array}{ll}
x^{k+1}= \tau\cdot\prox_{J(\cdot)}(\mathcal{A}^*y^k) \\
y^{k+1}=y^{k}+h(b-\mathcal{A}x^{k+1}).
\end{array} \right.
\end{eqnarray}
Similar to Theorem \ref{thm:cvg}, we can show:
\begin{theorem}
Set step size $h\in (0, \frac{2}{\tau\|\mathcal{A}\|^2})$ and $y^0=0$ in iteration (\ref{form5}). Let $\hat{x}$ be the unique minimizer to problem (\ref{Aug1}) and $\mathcal{W}$ be the solution set to problem \eqref{Dual1}. Then, $\lim_{k\rightarrow +\infty}x^k=\hat{x}$, and there exists a point $\bar{y}\in\mathcal{W}$ such that $\lim_{k\rightarrow +\infty} y^k=\bar{y}$.
\end{theorem}



\section*{Acknowledgements}
We would like to thank Professor Jian-Feng Cai (U. Iowa) for suggestions and corrections. The work of H. Zhang is supported in part by the Graduate School of NUDT under Fund of Innovation B110202, Hunan Provincial Innovation Foundation for Postgraduate CX2011B008, and NSFC grant 61201328.  The work of L. Cheng is supported in part by NSFC grants 61271014 and 61072118. The work of W. Yin is supported in part by NSF grants DMS-0748839 and ECCS-1028790.

\small{

}

\end{document}

%% file: macros.tex
\usepackage{amsmath,amssymb,amsthm}
\usepackage{latexsym,amsfonts,amscd,amsxtra,amstext}
\usepackage{algorithm,algorithmic}
\usepackage{url}

\usepackage{epsfig} 
\usepackage{threeparttable}
\usepackage{index}
\usepackage{subfigure}

\usepackage[normalem]{ulem} 

\newcommand{\cA}{{\mathcal{A}}}

\newcommand{\cY}{{\mathcal{Y}}}

\newcommand{\RR}{\mathbb{R}}

\newcommand{\prox}{{\mathbf{prox}}}

\DeclareMathOperator{\shrink}{shrink} 

\DeclareMathOperator*{\argmin}{arg\,min}

\newcommand{\bc}{\begin{center}}
\newcommand{\ec}{\end{center}}

\newcommand{\bdm}{\begin{displaymath}}
\newcommand{\edm}{\end{displaymath}}

\newcommand{\beq}{\begin{equation}}
\newcommand{\eeq}{\end{equation}}

\newcommand{\bfl}{\begin{flushleft}}
\newcommand{\efl}{\end{flushleft}}

\newcommand{\bt}{\begin{tabbing}}
\newcommand{\et}{\end{tabbing}}

\newcommand{\beqn}{\begin{eqnarray}}
\newcommand{\eeqn}{\end{eqnarray}}

\newcommand{\beqs}{\begin{align*}} 
\newcommand{\eeqs}{\end{align*}}  

\newcommand{\st}{\mbox{subject to}}

\newtheorem{theorem}{Theorem}

\newtheorem{definition}{Definition}

\newtheorem{remark}{Remark}
\newtheorem{lemma}{Lemma}